\documentclass[a4paper,10pt]{article}
\usepackage[utf8x]{inputenc}
\usepackage{tracefnt,amsmath,tabu,array}
\usepackage{amssymb,graphicx,setspace,amsfonts,amsbsy}
\usepackage{pifont,latexsym,ifthen,amsthm,rotating,calc,textcase,booktabs,cancel,slashed}

\newtheorem{theorem}{Theorem}[section]
\newtheorem{lemma}[theorem]{Lemma}

\newtheorem{proposition}[theorem]{Proposition}

\newtheorem{remark}[theorem]{Remark}
\newcommand{\filledbox}{\leavevmode
  \hbox to.77778em{%
  \hfil\vbox to.675em{\hrule width.6em height.6em}\hfil}}

\newcommand{\Rm}{{\mathbb R}}

\newcommand{\eps}{\varepsilon}

\begin{document}
\tabulinesep=1.0mm
\title{Exterior scattering of non-radial solutions to energy subcritical wave equations\footnote{MSC classes: 35L05, 35L71; the author is financially supported by National Natural Science Foundation of China Programs 12071339, 11771325.}}

\author{Ruipeng Shen\\
Centre for Applied Mathematics\\
Tianjin University\\
Tianjin, China}

\maketitle

\begin{abstract}
  We consider the defocusing, energy subcritical wave equation $\partial_t^2 u - \Delta u = -|u|^{p-1} u$ in dimension $d \in \{3,4,5\}$ and prove the exterior scattering of solutions if $3\leq d \leq 5$ and $1+6/d<p<1+4/(d-2)$. More precisely, given any solution with a finite energy, there exists a solution $u_L$ to the homogeneous linear wave equation, so that the following limit holds
 \[
 \lim_{t\rightarrow +\infty} \int_{|x|>t+R} |\nabla_{x,t} u(x,t)- \nabla_{x,t} u_L(x,t)|^2 dx = 0
\]
 for any fixed real number $R$. This generalize the previously known exterior scattering result in the radial case. 
\end{abstract}

\section{Introduction}
We consider the defocusing, energy subcritical wave equation ($d\geq 3$, $p<1+4/(d-2)$)
\[
 \left\{\begin{array}{ll} \partial_t^2 u - \Delta u = -|u|^{p-1} u, & (x,t) \in \Rm^d \times \Rm; \\
 u|_{t=0} =u_0 \in \dot{H}^1 \cap L^{p+1}(\Rm^d); & \\
 u_t |_{t=0} = u_1 \in L^2(\Rm^d). & \end{array}\right.\quad (CP1)
\]
The existence and uniqueness of local solutions follows a combination of suitable Strichartz estimates (see Ginibre-Velo \cite{strichartz}, for instance) and a fixed-point argument. Please see Kapitanski \cite{loc1} and Lindblad-Sogge \cite{ls} for more details about this kind of argument. We focus on the global behaviours in this work. 

\paragraph{Energy critical case} The case with an energy critical nonlinearity $p = p_e = 1 + 4/(d-2)$ has been extensively studied by many mathematicians in the last few decades of the 20th century. It has been proved that any solutions with initial data $(u_0,u_1)\in \dot{H}^1 \times L^2(\Rm^d)$ must exist for all time $t \in \Rm$ and scatter. By scattering we mean that there exist free waves $v^+, v^-$ (i.e. solutions to the homogeneous linear wave equation $\partial_t^2 v - \Delta v = 0$), so that 
\[
 \lim_{t\rightarrow \pm \infty} \left\|(u(\cdot,t), u_t(\cdot,t))-(v^\pm (\cdot,t), v_t^\pm (\cdot,t))\right\|_{\dot{H}^1 \times L^2(\Rm^d)} = 0. 
\] 
Please see, for example, Ginibre-Soffer-Velo \cite{locad1}, Grillakis \cite{mg1,mg2}, Kapitanski \cite{continuousL}, Nakanishi \cite{enscatter1, enscatter2}, Pecher \cite{local1}, Shatah-Struwe \cite{ss2}, Struwe \cite{struwe} for more details. 

\paragraph{Energy subcritical/supercritical case} The case with energy subcritical exponent $p<p_e$ or supercritical exponent $p>p_e$ seems to be more difficult. It is conjectured that any solution with initial data $(u_0,u_1)$ in the critical Sobolev space $\dot{H}^{s_p} \times \dot{H}^{s_p-1}$ must exist globally and scatter. Here $s_p = d/2 - 2/(p-1)$. Although this conjecture is still an open problem, there are many related scattering results. Roughly speaking, we may divide these results into two categories.

\paragraph{Conditional scattering} There are many works proving that if the critical Sobolev norm of a solution is bounded in its whole maximal lifespan, then the solution must scatter, for different dimensions $d$ and ranges of $p$. Please refer to Duyckaerts et al. \cite{dkm2}, Kenig-Merle \cite{km}, Killip-Visan \cite{kv2} (dimension 3), Killip-Visan \cite{kv3} (all dimensions) for energy supercritical case and Dodson-Lawrie \cite{cubic3dwave},  Dodson et al. \cite{nonradial3p5}, Shen \cite{shen2} (dimension 3),  Rodriguez \cite{sub45} (dimension 4 and 5) for energy subcritical case. All of the works mentioned above utilize the compactness-rigidity argument, which was first introduced by Kenig-Merle \cite{kenig, kenig1} in order to study the energy critical, focusing wave and Schr\"{o}dinger equations, and deal with both defocusing and focusing cases in the same way.   

\paragraph{Better initial data}  We may also prove the scattering in the energy subcritical case under suitable assumptions on the initial data, which is typically stronger than a finite critical Sobolev norm. Dodson \cite{claim1} proves the global existence and scattering of solutions in the conformal case of dimension 3 ($d=p=3$) if the initial data $(u_0,u_1) \in \dot{H}^{1/2} \times \dot{H}^{-1/2}$ are radial. There are also many scattering results assuming that the initial data are contained in a weighted Sobolev space. For example, Ginibre and Velo \cite{conformal2} apply conformal conservation laws and prove the scattering if 
\[
  \int_{\Rm^d} \left[(|x|^2+1) (|\nabla u_0 (x)|^2 + |u_1(x)|^2) + |u_0(x)|^2 \right] dx < \infty.
\]
Yang \cite{yang1} considers the energy momentum tensor and its associated currents and proves the scattering under a weaker assumption on initial data
 \begin{equation}
  \int_{\Rm^d} (1+|x|)^\kappa \left(\frac{1}{2}|\nabla u_0(x)|^2 + \frac{1}{2}|u_1(x)|^2 + \frac{1}{p+1}|u_0(x)|^{p+1}\right) < +\infty, \label{def weighted energy}
 \end{equation}
with $p$ and $\kappa$ satisfying 
\begin{align*}
 &\frac{1+\sqrt{d^2+4d-4}}{d-1} < p < p_e(d),& &\kappa> \max\left\{\frac{4}{p-1}-d+2,1\right\}.&
\end{align*}
Recently the author introduces the inward/outward energy theory and further improves the scattering theory when $1+4/(d-1)\leq p< 1+4/(d-2)$. More precisely, we still assume that $(u_0,u_1)$ satisfy \eqref{def weighted energy} but with a smaller lower bound of $\kappa$: (Please see \cite{shenenergy, shenhr} for radial cases and \cite{shen3dnonradial, shenhd} for non-radial cases)
\begin{align*}
   &\kappa >\kappa_1(d,p) = \frac{(d+2)(d+3)-(d+3)(d-2)p}{(d-1)(d+3)-(d+1)(d-3)p};& &\hbox{(Non-radial case)}& \\
   &\kappa>\kappa_2(d,p) = \frac{4-(d-2)(p-1)}{p+1};& &\hbox{(Radial case)}&
\end{align*}

\begin{remark}
It has been known many years ago that in the energy subcritical case a combination of the energy conservation law with a suitable local theory leads to the global existence of solutions as long as the initial data come with a finite energy. Please see, for example, Gibibre-Velo \cite{globalsubcritical} for more details. 
\end{remark} 

\paragraph{Topics of this work} In this short article we consider the exterior scattering of solutions to (CP1) in the energy subcritical case. More precisely, we consider whether there exists a free wave $u_L$, so that the limit (For convenience we use the notation $\nabla_{x,t} u= (\nabla_x u, \partial_t u)$)
\[
 \lim_{t\rightarrow +\infty} \int_{|x|>t+R} |\nabla_{x,t} u(x,t)-\nabla_{x,t} u_L(x,t)|^2  dx = 0
\]
holds for any fixed real number $R$. In the radial case, we may apply method of characteristic lines and verify that the exterior scattering happens whenever the solution comes with a finite energy.  Please see \cite{shenenergy, shenhr} for more details. In the non-radial case, however, the exterior scattering of an arbitrary finite-energy solution has not been proved in the energy subcritical case, as far as the author knows. The inward/outward energy theory mentioned above shows that ($r=|x|$ is the radius)
\[
 \lim_{t \rightarrow +\infty} \int_{\Rm^d} \left(\left|u_r+\frac{d-1}{2}\cdot \frac{u}{|x|} + u_t \right|^2 + \frac{(d-1)(d-3)}{16}\cdot \frac{|u|^2}{|x|^2} + |\slashed{\nabla} u|^2 \right) dx = 0.
\]
Radiation fields (Theorem \ref{radiation}) show that a free wave $u$ must satisfy the identity above as well. Thus the solutions to both the defocusing and free wave equation share some asymptotic behaviours. However, the information given by inward/outward energy theory is not sufficient to guarantee the scattering of solutions, even in the exterior region $\{(x,t): |x| > t+R\}$. In this work we combine energy flux formula and space-time cut-off techniques to prove the exterior scattering of all finite-energy solutions. Now let us give the main result of this work in details.
\begin{theorem} \label{main result}
Assume that the dimension $d$ and exponent $p$ satisfy $3\leq d \leq 5$ and $1+6/d < p < 1+4/(d-2)$. If $u$ is a solution to (CP1) with a finite energy, then there exists a finite-energy free wave $u_L (x,t)$ so that for any $R\in \mathbb{R}$ the following limit holds
\[
 \lim_{t\rightarrow +\infty} \int_{|x|>t+R} |\nabla_{x,t} u(x,t)-\nabla_{x,t} u_L(x,t)|^2  dx = 0.
\]
\end{theorem}

\begin{remark}
 Exterior scattering is clearly a weaker version of the scattering results in the whole space, and usually a first step to understand the asymptotic behaviour of solutions. For example, Duychaerts, Kenig and Merle prove exterior scattering of bounded solutions to energy critical, focusing wave equation in their work \cite{dkm3}. The first two authors then use this result to discuss soliton resolution of solutions in a subsequent joint work with Jia \cite{djknonradial}. 
\end{remark}

\paragraph{The idea} First of all, if $R$ is sufficiently large, most of the energy stays inside the light cone $|x|=t+R$. (see Lemma \ref{global tail estimate}) Thus it suffices to show the solution $u$ scatters in any cone shell $\Omega = \{(x,t): R_1+t<|x|<R_2+t\}$. Here $R_1<R_2$ are arbitrary constants. We break $u$ into two parts 
\[
 u(x,t) = v_T(x,t) + w_T(x,t), \quad (x,t) \in \Omega_T \doteq \Omega \cap (\Rm^d \times [T,+\infty)).
\]
They satisfy the wave equations $(\partial_t^2  - \Delta) v_T  = -|u|^{p-1}u$ and $(\partial_t^2 -\Delta) w_T = 0$ in $\Omega_T$, respectively. Roughly speaking, the function $w_T$ represents the waves travelling inside the cone $|x|=R_2+t$ after the time $t=T$. One may check that $v_T$ gradually becomes negligible as $T\rightarrow +\infty$ by the energy flux formula. In addition, the energy flux formula also implies that $\|u\|_{L^{p+1}(\Omega_T)}\rightarrow 0$ as $T\rightarrow +\infty$. Combining this fact with suitable Strichartz estimates and applying a continuity argument, we obtain $v_T \in L_t^q L_x^r (\Omega_T)$ for suitable constants $q,r$ and 
\[
  \lim_{T\rightarrow +\infty} \left\|\chi_\Omega |u|^{p-1} u\right\|_{L^1 L^2([T,+\infty)\times \Rm^d)} = 0. 
\]
Here $\chi_\Omega$ is the characteristic function of the region $\Omega$. Strichartz estimates then give the scattering of solutions in $\Omega$. More details can be found in later sections. 

\begin{remark} Let $(d,p)$ be as in the main theorem. If a solution $u$ to the focusing wave equation $\partial_t^2 u - \Delta u = +|u|^{p-1} u$ is defined for all $t\geq 0$ so that
\begin{itemize}
\item the absolute energy flux 
 \[
  g(R) = \int_{|x|=t+R, t\geq 0} \left(|\slashed{\nabla} u(x,t)|^2 + |u_r(x,t)+u_t(x,t)|^2 + |u(x,t)|^{p+1} \right) dS 
 \]
 is a local integrable function of $R \in \Rm$; 
\item The upper limit of absolute energy outside a light cone $|x|=t+R$ defined by
\[
  h(R) = \limsup_{t\rightarrow +\infty} \int_{|x|>t+R} \left(\frac{1}{2}|\nabla u(x,t)|^2 + \frac{1}{2}|u(x,t)|^2 + \frac{1}{p+1}|u(x,t)|^{p+1}\right) dx
\]
is a bounded function of $R\in \Rm$ and satisfies $\displaystyle \lim_{R\rightarrow +\infty} h(R) = 0$;
\end{itemize} 
then we may follow the same argument as in this work to prove the exterior scattering of $u$. 
\end{remark}

\section{Preliminary Results}
\begin{lemma} \label{global tail estimate}
 Let $u$ be a finite-energy solution to (CP1). Then we have the following limit 
 \[
  \lim_{R\rightarrow +\infty} \sup_{t\geq 0} \int_{|x|>t+R} \left(\frac{1}{2}|\nabla u(x,t)|^2 + \frac{1}{2}|u_t(x,t)|^2+\frac{1}{p+1}|u(x,t)|^{p+1}\right) dx = 0.
 \]
\end{lemma}
\begin{proof}
 By energy flux formula we have ($t \geq 0$)
\begin{align*}
 \int_{|x|>t+R} & \left(\frac{1}{2}|\nabla u(x,t)|^2 + \frac{1}{2}|u_t(x,t)|^2+\frac{1}{p+1}|u(x,t)|^{p+1}\right) dx \\
 & \leq \int_{|x|>R} \left(\frac{1}{2}|\nabla u_0 (x)|^2 + \frac{1}{2}|u_1(x)|^2+\frac{1}{p+1}|u_0(x)|^{p+1}\right) dx.
\end{align*}
The latter clearly converges to zero as $R\rightarrow +\infty$. 
\end{proof}
\begin{remark} \label{global free tail estimate}
 A similar argument shows that a finite-energy free wave $u_L$ satisfies a similar limit
\[
  \lim_{R\rightarrow +\infty} \sup_{t\geq 0} \int_{|x|>t+R} \left(\frac{1}{2}|\nabla u_L (x,t)|^2 + \frac{1}{2}|\partial_t u_L (x,t)|^2 \right) dx = 0.
 \] 
\end{remark}

\paragraph{Strichartz estimates} The generalized Strichartz estimates plays a key role in the local theory. The following version comes from Ginibre-Velo \cite{strichartz}. 

\begin{proposition}[Strichartz estimates]\label{Strichartz estimates} 
 Let $2\leq q_1,q_2 \leq \infty$, $2\leq r_1,r_2 < \infty$ and $\rho_1,\rho_2,s\in \Rm$ be constants with
 \begin{align*}
  &\frac{2}{q_i} + \frac{d-1}{r_i} \leq \frac{d-1}{2},& &(q_i,r_i)\neq \left(2,\frac{2(d-1)}{d-3}\right),& &i=1,2;& \\
  &\frac{1}{q_1} + \frac{d}{r_1} = \frac{d}{2} + \rho_1 - s;& &\frac{1}{q_2} + \frac{d}{r_2} = \frac{d-2}{2} + \rho_2 +s.&
 \end{align*}
 Assume that $u$ is the solution to the linear wave equation
\[
 \left\{\begin{array}{ll} \partial_t u - \Delta u = F(x,t), & (x,t) \in \Rm^d \times [0,T];\\
 u|_{t=0} = u_0 \in \dot{H}^s; & \\
 \partial_t u|_{t=0} = u_1 \in \dot{H}^{s-1}. &
 \end{array}\right.
\]
Then we have
\begin{align*}
 \left\|\left(u(\cdot,T), \partial_t u(\cdot,T)\right)\right\|_{\dot{H}^s \times \dot{H}^{s-1}} & +\|D_x^{\rho_1} u\|_{L^{q_1} L^{r_1}([0,T]\times \Rm^d)} \\
 & \leq C\left(\left\|(u_0,u_1)\right\|_{\dot{H}^s \times \dot{H}^{s-1}} + \left\|D_x^{-\rho_2} F(x,t) \right\|_{L^{\bar{q}_2} L^{\bar{r}_2} ([0,T]\times \Rm^d)}\right).
\end{align*}
Here the coefficients $\bar{q}_2$ and $\bar{r}_2$ satisfy $1/q_2 + 1/\bar{q}_2 = 1$, $1/r_2 + 1/\bar{r}_2 = 1$. The constant $C$ does not depend on $T$ or $u$. 
\end{proposition}
\begin{remark}
 If $(q_1,r_1)$ and $s$ satisfy the conditions given in Proposition \ref{Strichartz estimates} with $\rho_1=0$, we call $(q_1,r_1)$ an $s$-admissible pair. 
\end{remark}

\paragraph{Radiation fields} The following theorem describes the asymptotic behaviour of free waves, the details and proof of which can be found in Duyckaerts et al. \cite{dkm3} and Friedlander \cite{radiation1, radiation2}. 
\begin{theorem}[Radiation fileds] \label{radiation}
Assume that $d\geq 3$ and let $u$ be a solution to the homogeneous linear wave equation $\partial_t^2 u - \Delta u = 0$ with initial data $(u_0,u_1) \in \dot{H}^1 \times L^2(\Rm^d)$. Then 
\[
 \lim_{t\rightarrow +\infty} \int_{\Rm^d} \left(|\slashed{\nabla} u(x,t)|^2 + \frac{|u(x,t)|^2}{|x|^2}\right) dx = 0
\]
 and there exists a function $G_+ \in L^2(\Rm \times \mathbb{S}^{d-1})$ so that
\begin{align*}
 \lim_{t\rightarrow +\infty} \int_0^\infty \int_{\mathbb{S}^{d-1}} \left|r^{\frac{d-1}{2}} \partial_t u(r\theta, t) - G_+(r-t, \theta)\right|^2 d\theta dr &= 0;\\
 \lim_{t\rightarrow +\infty} \int_0^\infty \int_{\mathbb{S}^{d-1}} \left|r^{\frac{d-1}{2}} \partial_r u(r\theta, t) + G_+(r-t, \theta)\right|^2 d\theta dr & = 0.
\end{align*}
In addition, the map $(u_0,u_1) \rightarrow \sqrt{2} G_+$ is a bijective isometry from $\dot{H}^2 \times L^2(\Rm^d)$ to $L^2 (\Rm \times \mathbb{S}^{d-1})$.
\end{theorem}

\begin{lemma} \label{coefficients arg}
 Let $(d,p)$ be coefficients as in the main theorem. Then there exist a $1$-admissible pair $(q,r)$ with $q<+\infty$ and constants $k_1,k_2>0$ so that 
\begin{align*}
 &\frac{k_1}{p+1} + \frac{k_2}{q} = 1;& &\frac{k_1}{p+1} + \frac{k_2}{r} = \frac{1}{2};& &k_1+k_2 = p.&
\end{align*}
\end{lemma}
\begin{proof}
 It is clear that the constants 
\begin{align*}
 &\frac{1}{q} = \frac{2}{dp-d-2};& &\frac{1}{r} = \frac{dp-2p-d}{2(pd-d-2)};& \\
 &k_1 = \frac{(p+1)(d+2+2p-dp)}{(d+2+2p-dp)+2};& &k_2 = \frac{pd-d-2}{(d+2+2p-dp)+2}&
\end{align*}
satisfy the conclusion of Lemma \ref{coefficients arg}. Our assumption on $d$ and $p$ guarantees that $(q,r)$ given above is indeed a $1$-admissible pair.
\end{proof}
\begin{remark}
Form geometric point of view, we are looking for an admissible pair $(q,r)$, so that the points $(1/(p+1), 1/(p+1))$, $(1/p,1/2p)$ and $(1/q,1/r)$ are on the same straight line, as shown in figure \ref{figure adpair}. H\"{o}lder inequality implies 
\[
 \||u|^{p-1}u\|_{L^1 L^2} \leq \|u\|_{L^{p+1} L^{p+1}}^{k_1} \|u\|_{L^q L^r}^{k_2}.
\]
\end{remark}

 \begin{figure}[h]
 \centering
 \includegraphics[scale=0.6]{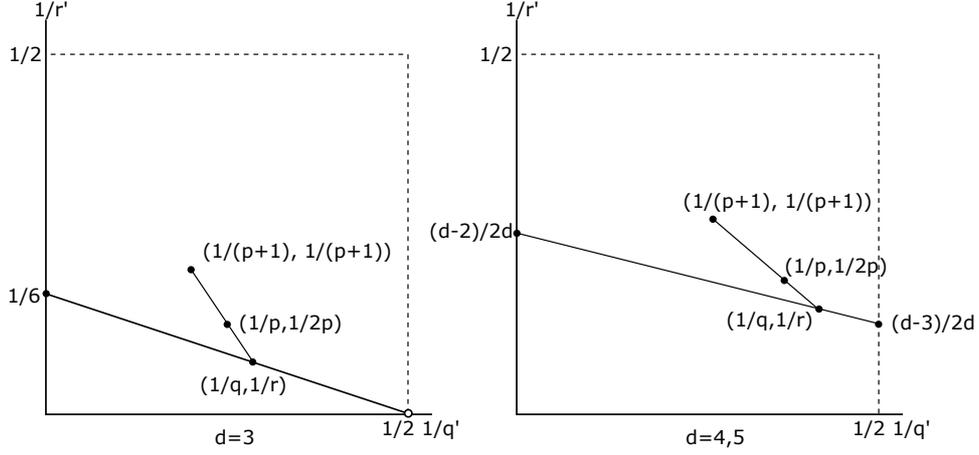}
 \caption{Illustration of line segment} \label{figure adpair}
\end{figure}

\section{Decomposition of solutions}
Let $R_1<R_2$ be fixed constants. We first fix a smooth cut-off function $\varphi: \Rm \rightarrow [0,1]$ satisfying
\[
 \varphi(r) = \left\{\begin{array}{ll} 1, & \hbox{if}\; r\geq 1;\\
 0, & \hbox{if}\; r\leq 1/2. \end{array}\right.
\]
Given a time $T\gg 1$, we define 
\[
 (w_{0,T} (x), w_{1,T}(x)) = \left\{\begin{array}{ll} (u(x,T),u_t(x,T)), & \hbox{if}\; |x|\geq T+R_2; \\
 \left(\varphi\left(\tfrac{x}{T+R_2}\right)u(x, |x|-R_2),0\right), & \hbox{if}\; |x|<T+R_2. \end{array} \right.
\]
Energy flux formula implies
\begin{equation} \label{energy flux bound}
 \int_{|x|=t+R_2} \left(\frac{1}{2}|\slashed{\nabla} u|^2 + \frac{1}{2}|(\partial_r +\partial_t)u|^2 + \frac{1}{p+1} |u|^{p+1}\right) dS \lesssim E,
\end{equation} 
thus we have
\[
 \|u(x,|x|-R_2)\|_{\dot{H}^1(\Rm^d)} \lesssim_1 E^{1/2}.
\]
A straightforward calculation shows 
\begin{equation} \label{boundedness of initial data w}
  \|(w_{0,T}, w_{1,T})\|_{\dot{H}^1 \times L^2(\Rm^d)} \lesssim_1 E^{1/2}, \quad \forall T \gg 1, 
\end{equation} 
and 
\begin{equation} \label{small initial data}
 \lim_{T\rightarrow +\infty} \int_{|x|<T+R_2} \left(|\nabla w_{0,T}(x)|^2 + |w_{1,T}(x)|^2\right) dx = 0.
\end{equation}
Next we use notations $\chi_1(x,t), \chi_2(x,t)$ to represent the characteristic functions of the regions $\{(x,t): |x|\geq t+R_2\}$ and $\{(x,t): t+R_1<|x|<t+R_2\}$, respectively. Given $T\gg 1$, let $w=w_T$ be the solution to the following linear wave equation
\[
 \left\{\begin{array}{ll} \partial_t^2 w - \Delta w = -\chi_1 (x,t) |u|^{p-1} u, & (x,t)\in \Rm^d \times [T,+\infty); \\
 w(x,T) = w_{0,T} (x); & x \in \Rm^d; \\
 w_t(x,T) = w_{1,T}(x); & x\in \Rm^d. \end{array}\right.
\] 
Please note that the initial data satisfy $(w_{0,T}(x),w_{1,T}(x)) = (u(x,T),u_t(x,T))$ if $|x|\geq T+R_2$. Thus by finite speed of propagation we have 
\begin{equation} \label{id w}
 (w(x,t), w_t(x,t)) = (u(x,t),u_t(x,t)), \quad |x|\geq t+R_2, \; t\geq T. 
\end{equation}
Similarly we define $v=v_T$ to be the solution to
\begin{equation} \label{equation v}
 \left\{\begin{array}{ll} \partial_t^2 v - \Delta v = -\chi_2 (x,t) |u|^{p-1} u, & (x,t)\in \Rm^d \times [T,+\infty); \\
 v(x,T) = v_{0,T}(x) \doteq u(x,T)-w_{0,T} (x); & x \in \Rm^d; \\
 v_t(x,T) = v_{1,T}(x) \doteq u_t(x,T)-w_{1,T}(x); & x\in \Rm^d. \end{array}\right.
\end{equation}
A combination of \eqref{boundedness of initial data w} and the energy conservation law implies that the initial data of $v$ is uniformly bounded
\begin{equation} \label{uniform upper bound initial v}
 \left\|(v_{0,T}, v_{1,T})\right\|_{\dot{H}^1\times L^2(\Rm^d)} \leq A, \quad \forall T \gg 1.
\end{equation}
Here $A\lesssim_1 E^{1/2}$ is a constant independent of $T$. In addition, $v+w$ solves the wave equation
\[
 (\partial_t^2 - \Delta) (v+w) = - [\chi_1(x,t) + \chi_2(x,t)] |u|^{p-1} u
\]
with the same data as $u$ at time $T$. We may utilize finite speed of propagation speed again and obtain 
\begin{equation} \label{v plus w equal u}
 (v(x,t)+w(x,t), v_t(x,t) +w_t(x,t)) = (u(x,t),u_t(x,t)), \quad |x|\geq t+R_1, \; t\geq T. 
\end{equation}
Please note that both solutions $v$ and $w$ depend on $T$. For simplicity we will omit the subscript $T$ in the argument below. 

\section{Estimation of Waves Going Inside}
Since the solution $w$ solves 
\[
 \partial_t^2 w - \Delta w = 0, \quad |x|<t+R_2;
\]
we may apply energy flux formula of the homogeneous linear wave equation and obtain ($t_1>T$)
\begin{align*}
 \int_{|x|<t_1+R_2} \left(|\nabla w(x,t_1)|^2 + |w_t(x,t_1)|^2\right) dx = & \int_{|x|<T+R_2} \left(|\nabla w(x,T)|^2 + |w_t(x,T)|^2\right) dx \\
 & + \frac{1}{\sqrt{2}} \int_{\Sigma(T,t_1)} \left(|\slashed{\nabla} w|^2 + |(\partial_r +\partial_t)w|^2 \right) dS\\
 = & \int_{|x|<T+R_2} \left(|\nabla w_{0,T}(x)|^2 + |w_{1,T}(x)|^2\right) dx\\
 & + \frac{1}{\sqrt{2}} \int_{\Sigma(T,t_1)} \left(|\slashed{\nabla} u|^2 + |(\partial_r +\partial_t)u|^2 \right) dS.
\end{align*}
Here $\Sigma(T,t_1) = \{(x,t): |x|=t +R_2, T<t<t_1\}$ is a part of the forward light cone $|x|=t+R_2$. We may substitute $w$ with $u$ in the surface integral above because of identity \eqref{id w}. We combine the identity above with the universal upper bound \eqref{energy flux bound} and the convergence of initial data \eqref{small initial data} to obtain
\begin{equation} \label{tail is small}
 \lim_{T\rightarrow +\infty} \sup_{t\geq T} \int_{|x|<t+R_2} \left(|\nabla w(x,t)|^2 + |w_t(x,t)|^2\right) dx = 0. 
\end{equation}
Next we define ($t_1>T>0$) 
\[
 (w_{0,T,t_1} (x), w_{1,T,t_1}(x)) = \left\{\begin{array}{ll} (w(x,t_1),w_t(x,t_1)), & \hbox{if}\; |x|\leq t_1+R_2; \\
 \left(u(x, |x|-R_2),0\right) & \hbox{if}\; |x|>t_1+R_2; \end{array} \right.
\]
Please note that the definitions in two regions coincide at the boundary $|x|=t_1+R_2$ according to \eqref{id w}. Combining \eqref{tail is small} and the fact that $u(x, |x|-R_2) \in \dot{H}^1(\Rm^d)$, we have 
\[
 \lim_{T\rightarrow +\infty} \sup_{t_1\geq T} \left\|(w_{0,T,t_1}, w_{1,T,t_1})\right\|_{\dot{H}^1 \times L^2(\Rm^d)} = 0. 
\]
By finite speed of propagation and Strichartz estimates, we obtain the following inequality for any $1$-admissible pair $(q,r)$:
\begin{align*}
 \|\chi_2 w\|_{L^q L^r ([T,t_1]\times \Rm^d)} = & \left\|\chi_2 \mathbf{S}_L(t-t_1) (w_{0,T,t_1}, w_{1,T,t_1})\right\|_{L^q L^r ([T,t_1]\times \Rm^d)}\\
 \leq & \left\|\mathbf{S}_L(t-t_1) (w_{0,T,t_1}, w_{1,T,t_1})\right\|_{L^q L^r ([T,t_1]\times \Rm^d)}\\
 \lesssim & \left\|(w_{0,T,t_1}, w_{1,T,t_1})\right\|_{\dot{H}^1 \times L^2(\Rm^d)}.
\end{align*}
Therefore we have 
\begin{equation} \label{tail is small Strichartz}
 \lim_{T\rightarrow +\infty} \|\chi_2 w\|_{L^q L^r ([T,+\infty)\times \Rm^d)} = 0.
\end{equation}

\section{Scattering in Cone Shells}
Let $(q,r)$ and $k_1,k_2$ be constants as in Lemma \ref{coefficients arg}. We recall equation \eqref{equation v}, apply Strichartz estimates and obtain ($t_1>T\gg 1$)
\[
 \|v\|_{L^q L^r ([T,t_1]\times \Rm^d)} \lesssim  \|(v_{0,T},v_{1,T})\|_{\dot{H}^1 \times L^2} + \|\chi_2 |u|^{p-1}u\|_{L^1 L^2([T,t_1]\times \Rm^d)} < +\infty. 
\]
We may also utilize the uniform upper bound \eqref{uniform upper bound initial v} and the H\"{o}lder inequality
\[
 \|\chi_2 |u|^{p-1} u\|_{L^1 L^2([T,t_1]\times \Rm^d)} \leq \|\chi_2 u\|_{L^{p+1} L^{p+1}([T,t_1]\times \Rm^d)}^{k_1} \|\chi_2 u\|_{L^q L^r([T,t_1]\times \Rm^d)}^{k_2},
\]
in the inequality above and obtain
\begin{align*}
 \|v\|_{L^q L^r ([T,t_1]\times \Rm^d)} \lesssim & A + \|\chi_2 u\|_{L^{p+1} L^{p+1}([T,t_1]\times \Rm^d)}^{k_1} \|\chi_2 u\|_{L^q L^r([T,t_1]\times \Rm^d)}^{k_2} \\
 \lesssim & A + \|u\|_{L^{p+1}(\Omega(T, t_1))}^{k_1} \left(\|\chi_2 v\|_{L^q L^r([T,t_1]\times \Rm^d)}+\|\chi_2 w\|_{L^q L^r([T,t_1]\times \Rm^d)}\right)^{k_2}.
\end{align*}
Here $\Omega(T, t_1) = \{(x,t): t+R_1<|x|<t+R_2, T<t<t_1\}$. In this region we have $u=v+w$ by \eqref{v plus w equal u}. Next we substitute the symbol $\lesssim$ with an explicit constant in the inequality above for the convenience of further discussion: 
\[
 \|v\|_{L^q L^r ([T,t_1]\times \Rm^d)} \leq C_1 A + C_1 \|u\|_{L^{p+1}(\Omega(T, t_1))}^{k_1} \left(\|\chi_2 v\|_{L^q L^r([T,t_1]\times \Rm^d)}+\|\chi_2 w\|_{L^q L^r([T,t_1]\times \Rm^d)}\right)^{k_2}. 
\]
The constant $C_1$ is solely determined by $d, p$ thus independent to $T, t_1$. By the energy flux formula, we also have
\[
 \iint_{t+R_1<|x|<t+R_2} |u(x,t)|^{p+1} dx dt \lesssim_1 (R_2-R_1) E.
\]
A combination of this fact with \eqref{tail is small Strichartz} implies that given any $\eps>0$, the following inequality always holds as long as $t_1>T$ and $T\geq T_1(\eps)$ is sufficiently large:
\begin{align*}
 \|v\|_{L^q L^r ([T,t_1]\times \Rm^d)} & \leq C_1 A + C_1 \eps \left(\|\chi_2 v\|_{L^q L^r([T,t_1]\times \Rm^d)} + \eps \right)^{k_2}\\
 & \leq C_1 A + C_1 \eps \left(\|v\|_{L^q L^r([T,t_1]\times \Rm^d)} + \eps \right)^{k_2}.
\end{align*}
We may choose a constant $\eps>0$ so that 
\[
 2C_1 A > C_1A + C_1 \eps \left(2C_1A + \eps \right)^{k_2}. 
\]
A continuity argument then shows that if $T\geq T_1(\eps)$, then 
\begin{equation} 
 \|v\|_{L^q L^r([T,t_1]\times \Rm^d)} < 2C_1 A, \; \forall t_1>T\quad \Rightarrow \quad \|v\|_{L^q L^r([T,+\infty)\times \Rm^d)} \leq 2C_1A.
\end{equation}
Therefore we have 
\begin{align*}
& \|\chi_2 |u|^{p-1}u\|_{L^1 L^2([T_1, +\infty)\times \Rm^d)}\\
  \leq & \|\chi_2 u\|_{L^{p+1} L^{p+1}([T_1,+\infty)\times \Rm^d)}^{k_1} \|\chi_2 u\|_{L^q L^r([T_1,+\infty)\times \Rm^d)}^{k_2}\\
  \leq &\|u\|_{L^{p+1}(\Omega(T_1, +\infty))}^{k_1} \left(\|\chi_2 v\|_{L^q L^r([T_1,+\infty)\times \Rm^d)}+\|\chi_2 w\|_{L^q L^r([T_1,+\infty)\times \Rm^d)}\right)^{k_2} \\
  < & +\infty. 
\end{align*}
Next we let $v_L(x,t) = \mathbf{S}_L(t-T) (v_{0,T}, v_{1,T})$ be free waves. By Strichartz estimates
\[
 \limsup_{t\rightarrow +\infty} \left\|(v(\cdot,t), v_t(\cdot,t))-(v_L(\cdot,t), \partial_t v_L(\cdot,t))\right\|_{\dot{H}^1 \times L^2} \lesssim_1 \|\chi_2 |u|^{p-1}u\|_{L^1 L^2([T,+\infty)\times \Rm^d)}
\]
Thus we have
\begin{equation}
 \lim_{T\rightarrow +\infty} \limsup_{t\rightarrow +\infty} \left\|(v(\cdot,t), v_t(\cdot,t))-(v_L(\cdot,t), \partial_t v_L(\cdot,t))\right\|_{\dot{H}^1 \times L^2} = 0.
\end{equation}
Combining this with \eqref{v plus w equal u} and \eqref{tail is small}, we obtain 
\begin{equation}
 \lim_{T\rightarrow +\infty} \limsup_{t\rightarrow +\infty} \int_{t+R_1<|x|<t+R_2} \left(|\nabla u(x,t) - \nabla v_L(x,t)|^2 + |u_t(x,t)-\partial_t v_L(x,t)|^2\right) dx = 0. 
\end{equation}
By radiation fields, there exist a family of functions $G_{T}(R,\theta) \in L^2([R_1,R_2]\times \mathbb{S}^{d-1})$ so that 
\begin{align*}
 \lim_{t'\rightarrow +\infty} \int_{R_1}^{R_2} \int_{\mathbb{S}^{d-1}} \left|(R+t')^{\frac{d-1}{2}} \partial_t v_L ((R+t')\theta, t') - G_T (R, \theta)\right|^2 d\theta dR &= 0;\\
 \lim_{t'\rightarrow +\infty} \int_{R_1}^{R_2} \int_{\mathbb{S}^{d-1}} \left|(R+t')^{\frac{d-1}{2}} \partial_r v_L ((R+t')\theta, t') + G_T (R, \theta)\right|^2 d\theta dR & = 0;\\
 \lim_{t\rightarrow +\infty} \int_{t+R_1<|x|<t+R_2} |\slashed{\nabla} v_L(x,t)|^2 dx & = 0.
\end{align*}
Therefore we have 
\begin{align}
 \lim_{T\rightarrow +\infty} \limsup_{t'\rightarrow +\infty} \int_{R_1}^{R_2} \int_{\mathbb{S}^{d-1}} \left|(R+t')^{\frac{d-1}{2}} \partial_t u ((R+t')\theta, t') - G_T (R, \theta)\right|^2 d\theta dR &= 0; \label{limit uGT1} \\
 \lim_{T\rightarrow +\infty} \limsup_{t'\rightarrow +\infty} \int_{R_1}^{R_2} \int_{\mathbb{S}^{d-1}} \left|(R+t')^{\frac{d-1}{2}} \partial_r u ((R+t')\theta, t') + G_T (R, \theta)\right|^2 d\theta dR & = 0; \label{limit uGT2}\\
 \lim_{t\rightarrow +\infty} \int_{t+R_1<|x|<t+R_2} |\slashed{\nabla} u(x,t)|^2 dx & = 0. \nonumber
\end{align}
The first limit implies that $(R+t')^{\frac{d-1}{2}} \partial_t u ((R+t')\theta, t')$ is a Cauchy sequence in the space $L^2([R_1,R_2]\times \mathbb{S}^{d-1})$ as $t'\rightarrow +\infty$. Thus there exists a function $G \in L^2([R_1,R_2]\times \mathbb{S}^{d-1})$, so that 
\[
 \lim_{t'\rightarrow +\infty} \int_{R_1}^{R_2} \int_{\mathbb{S}^{d-1}} \left|(R+t')^{\frac{d-1}{2}} \partial_t u ((R+t')\theta, t') - G (R, \theta)\right|^2 d\theta dR = 0.
\]
Combining \eqref{limit uGT1} and \eqref{limit uGT2}, we also have 
\[
 \limsup_{t'\rightarrow +\infty} \int_{R_1}^{R_2} \int_{\mathbb{S}^{d-1}} \left|(R+t')^{\frac{d-1}{2}} \partial_t u ((R+t')\theta, t') + (R+t')^{\frac{d-1}{2}} \partial_r u ((R+t')\theta, t')\right|^2 d\theta dR = 0.
\]  
Thus 
\[
 \lim_{t'\rightarrow +\infty} \int_{R_1}^{R_2} \int_{\mathbb{S}^{d-1}} \left|(R+t')^{\frac{d-1}{2}} \partial_r u ((R+t')\theta, t') + G (R, \theta)\right|^2 d\theta dR = 0.
\]
Finally we have
\begin{align*}
 \|G\|_{L^2 ([R_1, R_2]\times \mathbb{S}^{d-1})}^2 & = \lim_{t'\rightarrow +\infty} \int_{R_1}^{R_2} \int_{\mathbb{S}^{d-1}} \left|(R+t')^{\frac{d-1}{2}} \partial_t u ((R+t')\theta, t')\right|^2 d\theta dR\\
 & = \lim_{t'\rightarrow +\infty} \int_{t'+R_1<|x|<t'+R_2} |u_t(x,t')|^2 dx \leq 2E.
\end{align*}
\section{Exterior Scattering}
The argument in the previous sections works for any $-\infty < R_1 < R_2 < +\infty$, thus there exists a function $G \in L_{loc}^2 (\Rm \times \mathbb{S}^{d-1})$ so that the following limits hold for any $R_1<R_2$
\begin{align*}
 \lim_{t\rightarrow +\infty} \int_{t+R_1}^{t+R_2} \int_{\mathbb{S}^{d-1}} \left|r^{\frac{d-1}{2}} \partial_t u (r\theta, t) - G (r-t, \theta)\right|^2 d\theta dr &= 0;\\
  \lim_{t\rightarrow +\infty} \int_{t+R_1}^{t+R_2} \int_{\mathbb{S}^{d-1}} \left|r^{\frac{d-1}{2}} \partial_r u (r\theta, t) + G (r-t, \theta)\right|^2 d\theta dr &= 0;\\
  \lim_{t\rightarrow +\infty} \int_{t+R_1<|x|<t+R_2} |\slashed{\nabla} u(x,t)|^2 dx & = 0.
\end{align*}
 By the universal upper bound of $\|G\|_{L^2 ([R_1, R_2]\times \mathbb{S}^{d-1})}$ given at the end of last section, we actually have $G \in L^2 (\Rm \times \mathbb{S}^{d-1})$. 
By radiation fields we may find a free wave $u_L$ with a finite energy, so that
\begin{align*}
 \lim_{t\rightarrow +\infty} \int_{0}^{+\infty} \int_{\mathbb{S}^{d-1}} \left|r^{\frac{d-1}{2}} \partial_t u_L (r\theta, t) - G (r-t, \theta)\right|^2 d\theta dr &= 0;\\
  \lim_{t\rightarrow +\infty} \int_{0}^{+\infty} \int_{\mathbb{S}^{d-1}} \left|r^{\frac{d-1}{2}} \partial_r u_L (r\theta, t) + G (r-t, \theta)\right|^2 d\theta dr &= 0;\\
  \lim_{t\rightarrow +\infty} \int_{\Rm^d} |\slashed{\nabla} u_L (x,t)|^2 dx & = 0.
\end{align*}
Combining the limits above we have
\begin{align*}
 \lim_{t\rightarrow +\infty} & \int_{t+R_1<|x|<t+R_2} \left(|\nabla u(x,t)-\nabla u_L(x,t)|^2 + |\partial_t u(x,t)-\partial_t u_L(x,t)|^2 \right) dx \\
  = & \lim_{t\rightarrow +\infty}  \int_{t+R_1}^{t+R_2} \int_{\mathbb{S}^{d-1}} \left(\left|\partial_r u (r\theta, t) - \partial_r u_L(r\theta, t)\right|^2 + \left|\partial_t u (r\theta, t) - \partial_t u_L(r\theta, t)\right|^2\right) r^{d-1} d\theta dr\\
  & + \lim_{t\rightarrow +\infty} \int_{t+R_1<|x|<t+R_2} |\slashed{\nabla} u(x,t) -\slashed{\nabla} u_L (x,t)|^2 dx = 0.
\end{align*}
By Lemma \ref{global tail estimate} and Remark \ref{global free tail estimate} we also have 
\begin{align*}
 \lim_{R\rightarrow +\infty} \sup_{t\geq 0} \int_{|x|>t+R} \left(|\nabla u(x,t)-\nabla u_L(x,t)|^2 + |\partial_t u(x,t)-\partial_t u_L(x,t)|^2 \right) dx = 0. 
\end{align*}
Thus we have the exterior scattering
\[
 \lim_{t\rightarrow +\infty} \int_{|x|>t+R} \left(|\nabla u(x,t)-\nabla u_L(x,t)|^2 + |\partial_t u(x,t)-\partial_t u_L(x,t)|^2 \right) dx = 0.
\]

\end{document}